\documentclass[11pt]{amsart}%
\usepackage{amsfonts}
\usepackage{amsmath}
\usepackage{amssymb}
\usepackage{graphicx}%
\providecommand{\U}[1]{\protect\rule{.1in}{.1in}}
\newtheorem{theorem}{Theorem}[section]
\theoremstyle{plain}

\newtheorem{proposition}{Proposition}[section]

\numberwithin{equation}{section}
\begin{document}
\title[Rigidity]{Rigidity for general semiconvex entire solutions to the sigma-2 equation}
\author{Ravi Shankar}
\author{Yu YUAN}
\address{University of Washington\\
Department of Mathematics, Box 354350\\
Seattle, WA 98195}
\email{shankarr@uw.edu, yuan@math.washington.edu}
\thanks{Both authors are partially supported by NSF grants.}

\begin{abstract}
We show that every general semiconvex entire solution to the sigma-2 equation
is a quadratic polynomial. A decade ago, this result was shown for almost
convex solutions.

\end{abstract}
\date{\today }
\maketitle

\section{Introduction}

\label{sec:Intro}

In this paper, we show that every general semiconvex entire solution in
$\mathbb{R}^{n}$ to the Hessian equation
\[
\sigma_{k}\left(  D^{2}u\right)  =\sigma_{k}\left(  \lambda\right)
=\sum_{1\leq i_{1}<\cdots<i_{k}\leq n}\lambda_{i_{1}}\cdots\lambda_{i_{k}}=1
\]
with $k=2$ must be quadratic. Here $\lambda_{i}s$ are the eigenvalues of the
Hessian $D^{2}u.$

\begin{theorem}
\label{thm:Liou} Let $u$ be a smooth semiconvex solution to $\sigma_{2}\left(
D^{2}u\right)  =1$ on $\mathbb{R}^{n}$ with $D^{2}u\geq-K~I$ for a large
$K>0.$ Then $u$ is quadratic.
\end{theorem}

Recall the classical Liouville theorem for the Laplace equation $\sigma
_{1}\left(  D^{2}u\right)  =\bigtriangleup u=1$ or
J\"{o}rgens-Calabi-Pogorelov theorem for the Monge-Amp\`{e}re equation
$\sigma_{n}\left(  D^{2}u\right)  =\det D^{2}u=1:$ all convex entire solutions
to those equations must be quadratic. Theorem 1.1 has been settled under an
almost convexity condition $D^{2}u\geq\left(  \delta-\sqrt{2/\left[  n\left(
n-1\right)  \right]  }\right)  I$ for general dimension in the joint work with
Chang [ChY]; and under the general semiconvexity condition $D^{2}u\geq-KI$ in
three dimensions by taking advantage of the special Lagrangian form of the
equation in this case [Y]. Assuming a super quadratic growth condition,
Bao-Chen-Ji-Guan [BCJG] demonstrated that all convex entire solutions to
$\sigma_{k}\left(  D^{2}u\right)  =1$ with $k=1,2,\cdots,n$ are quadratic
polynomials; and Chen-Xiang [CX] showed that all \textquotedblleft super
quadratic\textquotedblright\ entire solutions to $\sigma_{2}\left(
D^{2}u\right)  =1$ with $\sigma_{1}\left(  D^{2}u\right)  >0$ and $\sigma
_{3}\left(  D^{2}u\right)  \geq-K$ are also quadratic polynomials. Warren's
rare saddle entire solutions for the $\sigma_{2}\left(  D^{2}u\right)  =1$
case [W] confirm the necessity of the semiconvexity assumption. It was
\textquotedblleft guessed\textquotedblright\ in the 2009 paper [ChY] that
Theorem 1.1 should hold true.

The equation $\sigma_{2}(\kappa)=1$ prescribes the intrinsic scalar curvature
of a Euclidean hypersurface $(x,u(x))$ in $\mathbb{R}^{n}\times\mathbb{R}^{1}$
with extrinsic principal curvatures $\kappa=\left(  \kappa_{1},\cdots
,\kappa_{n}\right)  .$ The $\sigma_{2}$ function of the Schouten tensor arises
in conformal geometry, and complex $\sigma_{2}$-type equations arise from the
Strominger system in string theory.

Our current work, as well the previous ones [ChY] [Y], has been inspired by
Nitsche's classical paper [N], where the Legendre-Lewy transform was employed
to produced an elementary proof of J\"{o}rgens' rigidity for the two
dimensional Monge-Amp\`{e}re equation, and in turn, Bernstein's rigidity for
the two dimensional minimal surface equation.

The Legendre-Lewy transform of a general semiconvex solution satisfies a
uniformly elliptic, saddle equation. In the almost convex case [ChY], the new
equation becomes concave, thus Evans-Krylov-Safonov theory yields the
constancy of the bounded new Hessian, and in turn, the old one. To beat the
saddle case, one has to be \textquotedblleft lucky\textquotedblright. Recall
that, in general Evans-Krylov-Safonov fails as shown by the saddle
counterexamples of Nadirashvili-Vl\u{a}du\c{t} [NV]. Our earlier trace Jacobi
inequality, as an alternative log-convex vehicle, other than the maximum
eigenvalue Jacobi inequality, in deriving the Hessian estimates for general
semiconvex solutions in [SY], could rescue the saddleness. But the trace
Jacobi only holds for large enough trace of the Hessian. It turns out that the
trace added by a large enough constant satisfies the elusive Jacobi inequality
(Proposition 2.1)

Equivalently, the reciprocal of the shifted trace Jacobi quantity is
superharmonic, and it remains so in the new vertical coordinates under the
Legendre-Lewy transformation by a transformation rule (Proposition 2.2). Then
the iteration arguments developed in the joint work with Caffarelli [CaY] show
the \textquotedblleft vertical\textquotedblright\ solution is close to a
\textquotedblleft harmonic\textquotedblright\ quadratic at one small scale
(Proposition 3.1, two steps in the execution: the superharmonic quantity
concentrates to a constant in measure by applying Krylov-Safonov's weak
Harnack; a variant of the superharmonic quantity, as a quotient of symmetric
Hessian functions of the new potential, is very pleasantly concave and
uniformly elliptic, consequently, closeness to a \textquotedblleft
harmonic\textquotedblright\ quadratic is possible by Evans-Krylov-Safonov
theory), and the closeness improves increasingly as we rescale (this is a
self-improving feature of elliptic equations, no concavity/convexity needed).
Thus a H\"{o}lder estimate for the bounded Hessian is realized, and
consequently so is the constancy of the new and then the old Hessian. See
Section 3.

In closing, we remark that, in three dimensions, our proof provides a
\textquotedblleft pure\textquotedblright\ PDE way to establish the rigidity,
distinct from the geometric measure theory way used in the earlier work on the
rigidity for special Lagrangian equations [Y, Theorem 1.3].

\bigskip

\section{Shifted trace Jacobi inequality and superharmonicity under
Legendre-Lewy transform}

Taking the gradient of both sides of the quadratic Hessian equation%
\begin{equation}
F\left(  D^{2}u\right)  =\sigma_{2}\left(  \lambda\right)  =\frac{1}{2}\left[
\left(  \bigtriangleup u\right)  ^{2}-\left\vert D^{2}u\right\vert
^{2}\right]  =1, \label{Esigm2}%
\end{equation}
we have%
\begin{equation}
\bigtriangleup_{F}Du=0, \label{Egradient}%
\end{equation}
where the linearized operator is given by
\begin{equation}
\bigtriangleup_{F}=\sum_{i,j=1}^{n}F_{ij}\partial_{ij}=\sum_{i,j=1}%
^{n}\partial_{i}\left(  F_{ij}\partial_{j}\right)  ,\ \ \label{lin}%
\end{equation}
with%
\begin{equation}
\left(  F_{ij}\right)  =\bigtriangleup u\ I-D^{2}u=\sqrt{2+\left\vert
D^{2}u\right\vert ^{2}}\ I-D^{2}u>0. \label{F_ij}%
\end{equation}
Here without loss of generality, we assume $\bigtriangleup u>0$ in the
remaining. Otherwise the smooth Hessian $D^{2}u$ would be in the
$\bigtriangleup u<0$ branch of the equation (\ref{Esigm2}). Given the
semiconvexity condition, the conclusion in Theorem \ref{thm:Liou} would be
straightforward by Evans-Krylov-Safonov.

The gradient square $\left\vert \nabla_{F}v\right\vert ^{2}$ for any smooth
function $v$ with respect to the inverse \textquotedblleft
metric\textquotedblright\ $\left(  F_{ij}\right)  $ is defined as%
\[
\left\vert \nabla_{F}v\right\vert ^{2}=\sum_{i,j=1}^{n}F_{ij}\partial
_{i}v\partial_{j}v.
\]

\subsection{Shifted trace Jacobi inequality}

\begin{proposition}
Let $u$ be a smooth solution to $\sigma_{2}\left(  \lambda\right)  =1\ $with
$D^{2}u\geq-KI.$ Set $b=\ln\left(  \bigtriangleup u+J\right)  .$ Then we have%
\begin{equation}
\bigtriangleup_{F}b\geq\varepsilon\left\vert \nabla_{F}b\right\vert ^{2}
\label{Jac}%
\end{equation}
for $J=8nK/3$ and $\ \varepsilon=1/3.$
\end{proposition}

\begin{proof}
\textit{Step 1. Differentiation of the trace}

We derive the following formulas for function $b=\ln\left(  \sigma
_{1}+J\right)  =\ln\left(  \bigtriangleup u+J\right)  :$%
\begin{equation}
\left\vert \nabla_{F}b\right\vert ^{2}=\sum_{i}f_{i}\frac{\left(
\bigtriangleup u_{i}\right)  ^{2}}{\left(  \sigma_{1}+J\right)  ^{2}}
\label{gradientb}%
\end{equation}
and%
\begin{gather}
\bigtriangleup_{F}b=\label{lapb}\\
\frac{1}{\left(  \sigma_{1}+J\right)  }\left\{  6\sum_{i>j>k}u_{ijk}%
^{2}+\left[  3\sum_{i\neq j}u_{jji}^{2}+\sum_{i}u_{iii}^{2}-\sum_{i}\left(
1+\frac{f_{i}}{\sigma_{1}+J}\right)  \left(  \bigtriangleup u_{i}\right)
^{2}\right]  \right\} \nonumber
\end{gather}
at $x=p,$ where, without loss of generality, $D^{2}u\left(  p\right)  $ is
assumed to be diagonalized and $f\left(  \lambda\right)  =\sigma_{2}\left(
\lambda\right)  .$

Noticing (\ref{F_ij}), it is straightforward to have the identity
(\ref{gradientb}) and at $p$
\begin{equation}
\bigtriangleup_{F}b=\sum_{i=1}^{n}f_{i}\left[  \frac{\partial_{ii}%
\bigtriangleup u}{\left(  \sigma_{1}+{\small J}\right)  }-\frac{\left(
\partial_{i}\bigtriangleup u\right)  ^{2}}{\left(  \sigma_{1}+{\scriptsize J}%
\right)  ^{2}}\right]  . \label{E4thorder}%
\end{equation}

Next we substitute the fourth order derivative terms $\partial_{ii}%
\bigtriangleup u=\sum_{k=1}^{n}\partial_{ii}u_{kk}$ in the above by lower
order derivative terms. Differentiating equation (\ref{Egradient})
$\sum_{i,j=1}^{n}F_{ij}\partial_{ij}u_{k}=0$ and using (\ref{F_ij}),\ we
obtain at $p$
\begin{align*}
\sum_{i=1}^{n}f_{i}\partial_{ii}\bigtriangleup u  &  =\bigtriangleup_{F}%
u_{kk}=\sum_{i,j=1}^{n}F_{ij}\partial_{ij}u_{kk}=\sum_{i,j=1}^{n}-\partial
_{k}F_{ij}\partial_{ij}u_{k}\\
&  =\sum_{i,j=1}^{n}-\left(  \bigtriangleup u_{k}\ \delta_{ij}-u_{kij}\right)
u_{kij}=\sum_{i,j=1}^{n}\left[  u_{ijk}^{2}-\left(  \bigtriangleup
u_{k}\right)  ^{2}\right]  .
\end{align*}
Plugging the above identity in (\ref{E4thorder}), we have at $p$%
\[
\hspace*{-0.5in}\bigtriangleup_{F}b=\frac{1}{\left(  \sigma_{1}+J\right)
}\left[  \sum_{i,j,k=1}^{n}u_{ijk}^{2}-\sum_{k=1}^{n}\left(  \bigtriangleup
u_{k}\right)  ^{2}-\sum_{i=1}^{n}\frac{f_{i}\ \ \left(  \bigtriangleup
u_{i}\right)  ^{2}}{\sigma_{1}+J}\right]
\]
Regrouping those terms $u_{\heartsuit\spadesuit\clubsuit},$ $u_{\spadesuit
\spadesuit\heartsuit},$ $u_{\heartsuit\heartsuit\heartsuit},$ and
$\bigtriangleup u_{\heartsuit}$ in the last two expressions, we obtain
(\ref{lapb}).

Subtracting (\ref{gradientb})$\ast\varepsilon$ from (\ref{lapb}), we have%
\[
\left(  \bigtriangleup_{F}b-\varepsilon\left\vert \nabla_{F}b\right\vert
^{2}\right)  \left(  \sigma_{1}+J\right)  \geq3\sum_{i\neq j}u_{jji}^{2}%
+\sum_{i}u_{iii}^{2}-\sum_{i}\left(  1+\delta\frac{f_{i}}{\sigma_{1}%
+J}\right)  \left(  \bigtriangleup u_{i}\right)  ^{2}%
\]
with $\delta=1+\varepsilon.$

Fix $i$ and denote $t=\left(  u_{11i},\cdots,u_{nni}\right)  $ and $e_{i}$ the
$i^{\prime}$th basis vector in $\mathbb{R}^{n},$ then the $i^{\prime}$th term
above can be written as%
\begin{equation}
Q=3\left\vert t\right\vert ^{2}-2\left\langle e_{i},t\right\rangle
^{2}-\left(  1+\delta\frac{f_{i}}{\sigma_{1}+J}\right)  \left\langle \left(
1,\cdots,1\right)  ,t\right\rangle ^{2} \label{Q}%
\end{equation}

\textit{Step 2.\ Tangential projection}

Equation (\ref{Egradient}) at $p$ yields that $t$ is tangential to the level
set of the equation $\sigma_{2}\left(  \lambda\right)  =1,\ $ $\left\langle
Df,t\right\rangle =0.$ Then by projecting $e_{i}$ and $\left(  1,\cdots
,1\right)  $ to the tangential space,%
\[
E=\left(  e_{i}\right)  _{T}=e_{i}-\frac{f_{i}}{\left\vert Df\right\vert ^{2}%
}Df\ \ \text{and \ }L=\left(  1,\cdots,1\right)  _{T}=\left(  1,\cdots
,1\right)  -\frac{\left(  n-1\right)  \sigma_{1}}{\left\vert Df\right\vert
^{2}}Df.
\]
The coefficients of the two negative terms in the quadratic form (\ref{Q})%
\[
Q=3\left\vert t\right\vert ^{2}-2\left\langle E,t\right\rangle ^{2}-\left(
1+\delta\frac{f_{i}}{\sigma_{1}+J}\right)  \left\langle L,t\right\rangle ^{2}%
\]
decrease, as simple symmetric computation shows
\begin{gather}
\left\vert E\right\vert ^{2}=1-\frac{f_{i}^{2}}{\left\vert Df\right\vert ^{2}%
}<1,\ \left\vert L\right\vert ^{2}=1-\frac{2\left(  n-1\right)  }{\left\vert
Df\right\vert ^{2}}<1,\ \ \label{CoefReduced}\\
\text{and }E\cdot L=1-\frac{\left(  n-1\right)  \sigma_{1}f_{i}}{\left\vert
Df\right\vert ^{2}}.\nonumber
\end{gather}

\textit{Step 3.\ Two anisotropic and non-orthogonal directions}

We proceed to show that the quadratic form $Q$ is positive definite. When $t$
is perpendicular to both $E$ and $L,$ $Q=3\left\vert t\right\vert ^{2}\geq0.$
So we only need to deal with the anisotropic case, when $t$ is along $\left\{
E,L\right\}  $-space. The corresponding matrix of the quadratic form $Q$ is%
\[
Q=3I-2E\otimes E-\eta L\otimes L
\]
with $\eta=1+\delta\frac{f_{i}}{\sigma_{1}+J}=1+\left(  1+\varepsilon\right)
\frac{f_{i}}{\sigma_{1}+J}.$ The \textit{real} $\xi$-eigenvector equation for
(symmetric) $Q$ under \textit{non-orthogonal} basis $\left\{  E,L\right\}  $
is%
\[
\left(
\begin{array}
[c]{cc}%
3-2\left\vert E\right\vert ^{2} & -2E\cdot L\\
-\eta L\cdot E & 3-\eta\left\vert L\right\vert ^{2}%
\end{array}
\right)  \left(
\begin{array}
[c]{c}%
\alpha\\
\beta
\end{array}
\right)  =\xi\left(
\begin{array}
[c]{c}%
\alpha\\
\beta
\end{array}
\right)  ,
\]
where corresponding real eigenvalues%
\[
\xi=\frac{1}{2}\left(  tr\pm\sqrt{tr^{2}-4\det}\right)  \ \ \ \text{with}%
\]%
\[
tr=6-2\left\vert E\right\vert ^{2}-\eta\left\vert L\right\vert ^{2}%
\ \ \ \ \text{and\ \ \ }\det=9-6\left\vert E\right\vert ^{2}-3\eta\left\vert
L\right\vert ^{2}+2\eta\left[  \left\vert E\right\vert ^{2}\left\vert
L\right\vert ^{2}-\left(  E\cdot L\right)  ^{2}\right]  .
\]

Now by (\ref{CoefReduced})%
\begin{align}
tr  &  =6-2\left(  1-\frac{f_{i}^{2}}{\left\vert Df\right\vert ^{2}}\right)
-\left(  1+\delta\frac{f_{i}}{\sigma_{1}+J}\right)  \left(  1-\frac{2\left(
n-1\right)  }{\left\vert Df\right\vert ^{2}}\right) \nonumber\\
&  >3-\delta\frac{f_{i}}{\sigma_{1}+J}=\frac{\left(  3-\delta\right)
\sigma_{1}+\delta\lambda_{i}+3J}{\sigma_{1}+J}>0\ \label{trLower}%
\end{align}
for any $\delta\leq1.5$ and $J\geq0,$ given $\sigma_{1}=\sqrt{\left\vert
\lambda\right\vert ^{2}+2}>\left\vert \lambda_{i}\right\vert $ in the
nontrivial remaining case.

Next again by (\ref{CoefReduced})%
\begin{gather*}
\det=6\frac{f_{i}^{2}}{\left\vert Df\right\vert ^{2}}-3\delta\frac{\ f_{i}%
}{\sigma_{1}+{\small J}}+3\left(  1+\delta\frac{\ f_{i}}{\sigma_{1}+{\tiny J}%
}\right)  \underbrace{\frac{2\left(  n-1\right)  }{\left\vert Df\right\vert
^{2}}}\\
+2\left(  1+\delta\frac{\ f_{i}}{\sigma_{1}+{\tiny J}}\right)  \ \left[
\frac{2\left(  n-1\right)  \sigma_{1}f_{i}}{\left\vert Df\right\vert ^{2}%
}-\frac{nf_{i}^{2}}{\left\vert Df\right\vert ^{2}}-\underbrace{\frac{2\left(
n-1\right)  }{\left\vert Df\right\vert ^{2}}}\right] \\
>-3\delta\frac{f_{i}}{\sigma_{1}+J}+4\left(  1+\delta\frac{f_{i}}{\sigma
_{1}+J}\right)  \frac{\left(  n-1\right)  \sigma_{1}f_{i}}{\left\vert
Df\right\vert ^{2}}+\left[  6-2n\left(  1+\delta\frac{f_{i}}{\sigma_{1}%
+J}\right)  \right]  \frac{f_{i}^{2}}{\left\vert Df\right\vert ^{2}}.
\end{gather*}
Then for $\delta=1+\varepsilon=4/3,$ we have%
\begin{gather}
\det\cdot\frac{\left(  \sigma_{1}+{\small J}\right)  \left\vert Df\right\vert
^{2}}{f_{i}}\geq\nonumber\\
-3\delta\underset{\left\vert Df\right\vert ^{2}}{\underbrace{\left[  \left(
n-1\right)  \sigma_{1}^{2}-2\right]  }}+\left\{
\begin{array}
[c]{c}%
4\left(  \sigma_{1}+J+\delta f_{i}\right)  \left(  n-1\right)  \sigma_{1}+\\
\left[  \left(  6-2n\right)  \left(  \sigma_{1}+{\small J}\right)  -2n\delta
f_{i}\right]  f_{i}%
\end{array}
\right\} \nonumber\\
=\left\{
\begin{array}
[c]{c}%
6\delta+4\left(  n-1\right)  J\sigma_{1}+2\left(  3-n\right)
J\underset{\sigma_{1}-\lambda_{i}}{\underbrace{f_{i}}}\\
+\left(  n-1\right)  \left(  4-3\delta\right)  \sigma_{1}^{2}+\left[
2n\left(  2\delta-1\right)  +6-4\delta\right]  \sigma_{1}f_{i}-2n\delta
f_{i}\underset{\sigma_{1}-\lambda_{i}}{\underbrace{f_{i}}}%
\end{array}
\right\} \nonumber\\
\overset{\delta=4/3}{=}8+2\left(  n+1\right)  J\sigma_{1}+2\left(  n-3\right)
J\lambda_{i}+\frac{2\left(  n+1\right)  }{3}\sigma_{1}f_{i}+\frac{8}%
{3}n\lambda_{i}f_{i}\ \ \label{detLower}\\
>2\left(  n+1\right)  J\sigma_{1}+2\left(  n-3\right)  J\lambda_{i}+\frac
{8}{3}n\lambda_{i}f_{i}. \label{detLast}%
\end{gather}

Case $\lambda_{i}\geq0:$ (\ref{detLast}) is positive by the ellipticity
$f_{i}>0$ from (\ref{F_ij}).

Case $0>\lambda_{i}\geq-K:$
\[
(\ref{detLast})=2nJ\underset{\sqrt{2+\left\vert \lambda\right\vert ^{2}%
}+\lambda_{i}>0}{\underbrace{\left(  \sigma_{1}+\lambda_{i}\right)  }%
}-6J\lambda_{i}+2J\sigma_{1}+\frac{8}{3}n\lambda_{i}\underset{\sigma
_{1}-\lambda_{i}<2\sigma_{1}}{\underbrace{f_{i}}}\ >0
\]
if $J=8nK/3.$

Therefore, the quadratic form $Q$ is positive definite, and we have derived
the shifted Jacobi inequality (\ref{Jac}) in the semiconvex case.
\end{proof}

\noindent\textbf{Remark.} \textit{In three dimensions, the Jacobi inequality
(\ref{Jac}) still holds for any }$J\geq0$\textit{ and }$\varepsilon
=1/3$\textit{ without the semiconvexity assumption }$D^{2}u\geq-KI.$ Actually,
we only need to show that, in Step 3, (\ref{detLower}) with $\delta
=1+\varepsilon=4/3<1.5$ is also positive for negative $\lambda_{i}.$ We would
have the desired lower bound for (\ref{detLower})
\[
\det\cdot\frac{\left(  \sigma_{1}+{\small J}\right)  \left\vert Df\right\vert
^{2}}{f_{i}}>8f_{i}\left(  \frac{\sigma_{1}}{3}+\lambda_{i}\right)  >0,
\]
if we know $\lambda_{i}>-\sigma_{1}/3.$ Without loss of generality, we assume
$\lambda_{1}\geq\lambda_{2}\geq\lambda_{3}.$ Because $\lambda_{2}+\lambda
_{3}=f_{1}>0,$ only the smallest eigenvalue $\lambda_{3}$ could be negative.
In such a negative case $\lambda_{3}=\frac{1-\lambda_{1}\lambda_{2}}%
{\lambda_{1}+\lambda_{2}}$ with $\lambda_{1}\lambda_{2}>1,$ we do have
$\lambda_{i}>-\sigma_{1}/3$ or $\frac{\sigma_{1}}{-\lambda_{3}}>3.$ Because%
\begin{align*}
\frac{\sigma_{1}}{-\lambda_{3}}  &  =-1+\frac{\left(  \lambda_{1}+\lambda
_{2}\right)  ^{2}}{\lambda_{1}\lambda_{2}-1}\\
&  \geq-1+\frac{4\lambda_{1}\lambda_{2}}{\lambda_{1}\lambda_{2}-1}>3.
\end{align*}

Note that, in three dimensions, the Jacobi inequality for the log-convex
$b=\ln\bigtriangleup u=\ln\sqrt{2+\left\vert \lambda\right\vert ^{2}}$ (with
$\varepsilon=1/100$) was derived by Qiu [Q, Lemma 3] for solutions to
(\ref{Esigm2}) along with variable right hand side; and the Jacobi inequality
with $\varepsilon=1/3$ for the log-max $b=\ln\lambda_{\max}$ (with
$\varepsilon=1/3$) was derived for solutions to (\ref{Esigm2}) in [WY, Lemma 2.2].

In general dimensions, a Jacobi inequality for sufficiently large $b=\ln
u_{11},$ at points where $u_{11}=\lambda_{\max},$ was obtained for solutions
having $\sigma_{3}\left(  D^{2}u\right)  $ lower bound to (\ref{Esigm2}) along
with variable right hand side by Guan-Qiu [GQ, p.1650]; and another Jacobi
inequality for sufficiently large $b=\ln\lambda_{\max}$ was derived for
semiconvex solutions to (\ref{Esigm2}) in [SY, Proposition 2.1], as mentioned
in the introduction.

\subsection{Superharmonicity under Legendre-Lewy transform}

Set $\tilde{u}\left(  x\right)  =u\left(  x\right)  +\bar{K}\left\vert
x\right\vert ^{2}/2$ for our $K$-semiconvex entire solution $u$ and say,
$\bar{K}=J/n\ >K+1,$ where $J=8nK/3$ is from Proposition 2.1. The $\bar{K}%
$-convexity of $\tilde{u}$ ensures that the smallest canonical angle of the
\textquotedblleft Lewy-sheared\textquotedblright\ \textquotedblleft
gradient\textquotedblright\ graph is larger than $\pi/4.$ This means we can
make a well defined Legendre reflection about the origin,
\begin{equation}
(x,D\tilde{u}(x))=(Dw(y),y)\in\mathbb{R}^{n}\times\mathbb{R}^{n} \label{LL}%
\end{equation}
where $w(y)$ is the Legendre transform of $u+\frac{\bar{K}}{2}|x|^{2};$ see
[L]. Note that $y(x)=Du(x)+Kx$ is a diffeomorphism from $\mathbb{R}^{n}$ to
$\mathbb{R}^{n}$ and%
\[
0<D^{2}w=\left(  D^{2}u+\bar{K}I\right)  ^{-1}<I.
\]
More precisely, by [ChY, p.663] or [SY, (2.11)], the eigenvalues $\mu_{1}%
\leq\mu_{2}\leq\cdots\leq\mu_{n}$ of $D^{2}w$ satisfy%
\begin{equation}
0<\mu_{1}\leq c\left(  n\right)  <1\ \ \text{and \ }0<c\left(  n,K\right)
\leq\mu_{i}<1\ \ \text{for }i\geq2. \label{IEnewHess}%
\end{equation}
As shown in [ChY, p.663] or [SY, proof of Proposition 2.4], the equation
solved by the vertical coordinate Lagrangian potential $w(y)$,
\[
G(D^{2}w)=-F(D^{2}u)=-\sigma_{2}(-\bar{K}I+(D^{2}w)^{-1})=-1,
\]
is conformally, uniformly elliptic for $K$-convex solutions $u,$ in the sense
that for $H_{ij}:=\sigma_{n}(\mu(D^{2}w))G_{ij},$ the linearized operator
$H_{ij}\partial_{ij}$ of equation%
\begin{align}
0  &  =H\left(  D^{2}w\right)  =\sigma_{n}\left(  D^{2}w\right)  \left[
G\left(  D^{2}w\right)  +1\right] \nonumber\\
&  =-\sigma_{n-2}\left(  \mu\right)  +\underset{A_{1}}{\underbrace{\left(
n-1\right)  \bar{K}}}\ \sigma_{n-1}\left(  \mu\right)  -\underset{A_{2}%
}{\underbrace{\left[  \frac{n\left(  n-1\right)  }{2}\bar{K}^{2}-1\right]  }%
}\sigma_{n}\left(  \mu\right)  \label{Econformal}%
\end{align}
is uniformly elliptic:
\[
c(n,K)I\leq(H_{ij})=\sigma_{n}(\mu)(G_{ij})\leq C(n,K)I.
\]

\begin{proposition}
\label{prop:superharmonic}Let $u\left(  x\right)  $ be a smooth solution to
$\sigma_{2}\left(  \lambda\right)  =1\ $with $D^{2}u\geq-KI.$ Set%
\[
a\left(  y\right)  =\left(  \frac{1}{\mu_{1}}+\cdots+\frac{1}{\mu_{n}}\right)
^{-1/3}=\left[  \frac{\sigma_{n}}{\sigma_{n-1}}\left(  \mu\right)  \right]
^{1/3}%
\]
with $\mu_{i}$s being the eigenvalues of the Hessian $D^{2}w\left(  y\right)
$ of the Legendre-Lewy transform of $u\left(  x\right)  +\bar{K}\left\vert
x\right\vert ^{2}/2.$ Then we have%
\[
\bigtriangleup_{H}a\leq0.
\]

\end{proposition}

\begin{proof}
The trace Jacobi inequality (\ref{Jac}) in Proposition 2.1 with \ $J=n\bar{K}$
is equivalent to%
\[
\bigtriangleup_{F}\left(  \bigtriangleup u+J\right)  ^{-1/3}=\bigtriangleup
_{F}e^{-b/3}\leq0.
\]
Noticing that $\bigtriangleup u+n\bar{K}=\frac{1}{\mu_{1}}+\cdots+\frac{1}%
{\mu_{n}},$ and applying the transformation rule [SY, Proposition 2.3], we
immediately obtain the desired superharmonicity%
\[
\bigtriangleup_{H}a=\sigma_{n}\left(  \mu\right)  \bigtriangleup_{G}a\leq0.
\]

\end{proof}

\section{H\"{o}lder Hessian estimate for saddle equation and rigidity}

The Hessian bound $0<D^{2}w(y)\leq I$ ensures that establishing a local
$C^{2,\alpha}$ estimate for such solutions to (\ref{Econformal}) will prove,
by scaling, that $w(y)$ is a quadratic polynomial. By the iteration arguments
developed in [CaY] for such smooth PDEs $H(D^{2}w)=0$ with solutions
satisfying Hessian bounds, proving $C^{2,\alpha}$ regularity at a point, say
the origin, reduces to showing that $w(y)$ is close to a uniform quadratic
polynomial, namely:

\begin{proposition}
\label{prop:concentrate} Let $u\left(  x\right)  $ be a smooth solution to
$\sigma_{2}\left(  \lambda\right)  =1\ $with $D^{2}u\geq-KI$ in $\mathbb{R}%
^{n}.$ Let $w(y)$ be its Legendre-Lewy transform defined in (\ref{LL}) solving
(\ref{Econformal}) in $\mathbb{R}^{n}$ with $0<D^{2}w\leq I.$ Given any
$\theta>0,$ there exists small $\eta=\eta(n,K,\theta)>0$ and a quadratic
polynomial $P(y)$ whose coefficients only depend on $n,K,\theta$ such that
\[
\left\vert \frac{1}{\eta^{2}}w(\eta z)-P(z)\right\vert \leq\theta
\]
is valid for $|z|\leq1.$
\end{proposition}

\smallskip In the case that the level set $\{H\left(  D^{2}w\right)  =0\}$
were convex (in fact saddle from [ChY, p.661]), the alternative way in [CaY]
other than Evans-Krylov-Safonov is the following. The Laplacian $\Delta w(y)$
is a sub or supersolution of the linearized operator $\Delta_{H}%
=H_{ij}\,\partial^{2}/\partial y_{i}\partial y_{j}$ of $H(D^{2}w).$ The weak
Harnack inequality shows that $\Delta w(y)$ concentrates in measure at a level
$c$ on a small ball $B=B_{r}(0).$ Solving the equation $\Delta v=c$ on $B$
with $v=w$ on $\partial B$ furnishes the desired smooth approximation, which
is uniform by the ABP estimate. The Laplacian can be replaced with any
elliptic slice of the Hessian, such that the elliptic slice is a supersolution
of $\Delta_{H},$ and the corresponding elliptic equation both has
$C^{2,\alpha}$ interior regular solutions and allows for the ABP estimate.

\smallskip However, it is not clear if the saddle level set $\{H\left(
D^{2}w\right)  =0\}$ of (\ref{Econformal}) is any of trace-convex [CaY],
max-min [CC], or twisted [CS] [C], so it is not clear if there are good PDEs
which super-solve $\Delta_{H}.$ Now that the remarkable superharmonic quantity
$\sigma_{n}\left(  \mu\right)  /\sigma_{n-1}\left(  \mu\right)  $ in
Proposition 2.2 is available, the core method in [CaY, pg 687-690] becomes
more realistic.

\smallskip There is still one more hurdle to overcome. The superharmonic,
\textquotedblleft one-step\textquotedblright\ Hessian quotient $a^{3}%
=\sigma_{n}\left(  \mu\right)  /\sigma_{n-1}(\mu)$ is well known to be
concave, but not uniformly elliptic, because $\sigma_{n}\left(  \mu\right)  $
could be arbitrarily close to zero. This prevents applying
Evans-Krylov-Safonov theory. We resolve this by substituting the concentration
of $a$ into the \textquotedblleft conformal\textquotedblright\ equation
(\ref{Econformal}). This implies concentration of a better quantity. Observe
that equation (\ref{Econformal}) can be written as
\begin{equation}
q\left(  \mu\right)  :=\frac{\sigma_{n-1}(\mu)}{\sigma_{n-2}(\mu)}=\left[
A_{1}-A_{2}\frac{\sigma_{n}\left(  \mu\right)  }{\sigma_{n-1}\left(
\mu\right)  }\right]  ^{-1}. \label{Enew}%
\end{equation}
Thus, the concentration of the higher quotient $a^{3}=\sigma_{n}/\sigma_{n-1}$
implies concentration of the lower quotient $\sigma_{n-1}/\sigma_{n-2},$ which
is also a concave operator [L, Theorem 15.18]. The almost-convex case,
$D^{2}u\geq(-K+\delta)I$ for $K^{-2}=n\left(  n-1\right)  /2$ and any
$\delta>0$ considered in [ChY], corresponds to $A_{2}=0.$ There, it was shown
that (\ref{Econformal}) is uniformly elliptic for arbitrarily large $K,$ in
particular, the lower quotient $\sigma_{n-1}/\sigma_{n-2}=A_{1}^{-1}$ for
$K^{-2}=n\left(  n-1\right)  /2.$ For arbitrary $K,$ using the bound for $\mu$
in (\ref{IEnewHess}) and the result in [L, Theorem 15.18], we deduce the
uniform ellipticity of $q\left(  \mu\right)  $
\begin{equation}
\partial_{\mu_{i}}q\in\frac{\sigma_{n-1}\left(  \mu\right)  }{\sigma_{n-2}%
^{3}(\mu)}\sigma_{n-2,i}^{2}(\mu)\left[  c\left(  n\right)  ,1\right]
\subset\ \left[  c\left(  n,K\right)  ,C\left(  n,K\right)  \right]  .
\label{elliptic}%
\end{equation}

\begin{proof}
[Proof of Proposition 3.1]Given any small $\xi,\delta>0$, we denote
$a_{k}=\min_{B_{1/2^{k}}}a$ and define a \textquotedblleft bad
set\textquotedblright\ $E_{k}=\{y\in B_{1/2^{k}}:a>a_{k}+\xi\}$. Using
Krylov-Safonov's weak Harnack inequality [GT, Theorem 9.22] for supersolution
$a(y)$ from Proposition \ref{prop:superharmonic}, the iteration argument in
[CaY, pg 687-690] shows that there exists $k_{0}(n,K,\xi,\delta)$ large enough
such that $|E_{\ell}|<\delta|B_{1/2^{\ell}}|$ for some $\ell\in\lbrack
1,k_{0}].$ By applying the quadratic scaling $2^{2\ell}w(2^{-\ell}y)$ and
modifying the small constant $\eta$ in Proposition \ref{prop:concentrate}, we
may assume $\ell=0.$ This is without loss of generality and is possible, if we
assume $w\left(  0\right)  =0$ and $Dw\left(  0\right)  =0$ in the beginning.

Using uniform ellipticity
\eqref{elliptic}%
, let us extend $\sigma_{n-1}/\sigma_{n-2}$ to a uniformly elliptic concave
operator $q(\mu)$ outside the eigenvalue rectangle (\ref{IEnewHess}),\ $\mu
\in\lbrack0,c(n)]\times\lbrack c(n,K),C(n,K)]^{n-1}.$ The notation $q$ in
(\ref{Enew}) was abused for the sake of notation simplicity. Let $v(y)\in
C^{\infty}(B_{1})$ solve the concave equation $q\left[  v\right]  =q\left(
\mu(D^{2}v)\right)  =\left(  A_{1}-A_{2}a_{\ell}\right)  ^{-1}$ in $B_{1}$
with $v=w$ on $\partial B_{1}.$ Then from the quotient representation
\eqref{Enew}
for the equation that $w$ solves, the ABP estimate [GT, Theorem 9.1] yields on
$B_{1}$,
\begin{align*}
|w-v|  &  \leq C(n,K)\Vert q[w]-q[v]\Vert_{L^{n}(B_{1})}\\
&  \leq C(n,K)\delta+C(n,K)\left\Vert \frac{a^{3}-a_{\ell}^{3}}{(A_{1}%
-A_{2}\,a^{3})(A_{1}-A_{2}\,a_{\ell}^{3})}\right\Vert _{L^{n}(E_{\ell}^{c})}\\
&  \leq C(n,K)(\delta+\xi),
\end{align*}
where, in the last inequality, we used the boundedness of $\left(  A_{1}%
-A_{2}\,a\right)  ^{-1}$ via (\ref{Enew}) and (\ref{IEnewHess}). By
Evans-Krylov-Safonov theory applied to the smooth equation $q\left[  v\right]
=a_{\ell},$\ $v(y)$ has uniform interior estimates, so $v$ can be replaced by
its quadratic part at the origin up to a uniform $O(|x|^{3})$ term, which is
$O(r^{3})$ on $B_{r}(0)$. The conclusion of this proposition follows for
$\eta=r$ in a standard way by successively choosing the small constants
$\xi,\delta,r$ depending on the small parameter $\theta$ and $n,K.$
\end{proof}

\bigskip

\begin{proof}
[Proof of Theorem 1.1]As indicated in the beginning of Section 2, we only need
to handle the positive branch $\bigtriangleup u>0$ of the quadratic equation
$\sigma_{2}\left(  D^{2}u\right)  =1.\ $This is because the only other
possibility is that $D^{2}u$ is on the negative branch $\bigtriangleup u<0$ of
the still elliptic and concave equation $\sigma_{2}\left(  D^{2}u\right)  =1.$
Then the semiconvex solutions must have bounded Hessian, and consequently, the
conclusion in Theorem \ref{thm:Liou} is straightforward by Evans-Krylov-Safonov.

Now armed with Proposition 3.1, the initial closeness of $w$ to a
\textquotedblleft harmonic\textquotedblright\ quadratic on the unit ball, and
repeating the proof of Proposition 2 in [CaY] with the equation there replaced
by our smooth uniformly elliptic equation (\ref{Enew}), we see that the
closeness to \textquotedblleft harmonic\textquotedblright\ quadratics
accelerates. As in [CaY, p.692], we obtain that $D^{2}w$ is H\"{o}lder at the
origin. Similarly, one proves that $D^{2}w$ is H\"{o}lder in the half ball%
\[
\left[  D^{2}w\right]  _{C^{\alpha}\left(  B_{1/2}\right)  }\leq C\left(
n,K\right)  ,
\]
where $\alpha=\alpha\left(  n,K\right)  >0.$

By quadratic scaling \ $R^{2}w\left(  y/R\right)  ,$ we get%
\[
\left[  D^{2}w\right]  _{C^{\alpha}\left(  B_{R}\right)  }\leq\frac{C\left(
n,K\right)  }{R^{\alpha}}\longrightarrow0,\ \ \text{as }R\rightarrow\infty.
\]
We conclude that $D^{2}w$ is a constant matrix, and in turn, so is $D^{2}u.$
\end{proof}

\end{document}